\documentclass[12pt]{article}
\usepackage{amsmath, amsfonts,amsthm,amssymb,amsbsy,upref,color,graphicx,amscd,hyperref,makeidx,enumerate}
\usepackage{tikz}
\usepackage[active]{srcltx}
\usepackage[latin1]{inputenc}
\usepackage{enumerate}
\usepackage[english]{babel}
\usepackage{newlfont}
\usepackage{amsbsy}
\usepackage[english]{babel}
\usepackage[center]{caption2}
\usepackage{amssymb}
\usepackage{amsmath}
\usepackage{latexsym}
\usepackage{amsthm}
\usepackage{amsthm}
\theoremstyle{plain}
\newtheorem{thm}{Theorem}

\newtheorem{prop}[thm]{Proposition}
\newtheorem{nota}[thm]{Notation}
\newtheorem{rem}[thm]{Remark}
\newtheorem{defin}[thm]{Definition}

\newcommand{\R}{\mathbb{R}}

\newcommand{\N}{\mathbb{N}}

\usepackage{mathtools}
\def\multiset#1#2{\ensuremath{\left(\kern-.2em\left(\genfrac{}{}{0pt}{}{#1}{#2}\right)\kern-.2em\right)}}
\usepackage[center]{caption2}

\begin{document}

\title{Graphlike families of multiweights}
\author{Agnese Baldisserri, Elena Rubei}
\date{}
\maketitle

\begin{abstract}

Let ${\cal G}=(G,w) $
be a weighted graph, that is, a graph $G$ 
endowed with a function $w$ from the edge set of $G$ to the set of real numbers;
for any subset $S$ of the  vertex set  of $G$,  we define $D_S({\cal G})$ 
to be the minimum of the weights of the subgraphs of $G$ whose vertex set contains  $S$; 
we call $D_S({\cal G})$  a  multiweight of ${\cal G}$.

Let $X$ be a finite set and let $\{D_S\}_{S \subset X, \; \sharp S \geq 2}  $ be a family of positive real numbers. 
We find  necessary and sufficient conditions 
for the family  to be the family of multiweights of a positive-weighted graph with vertex set $X$. 
  Moreover we study the analogous problem  for trees.
Finally, 
we find a criterion to say if there exists a nonnegative-weighted tree
${\cal T}$  with leaf set $X$ and such that $D_S ({\cal T})=D_S $ for any 
$S \subset X$. 
\end{abstract}

\def\thefootnote{}
\footnotetext{ \hspace*{-0.36cm}

 }

\section{Introduction}

For any graph $G$, let $E(G)$, $V(G)$ and $L(G)$ 
 be respectively the set of the edges,   
the set of the vertices and  the set of the leaves of $G$.
A {\bf weighted graph} ${\cal G}=(G,w)$ is a graph $G$ 
endowed with a function $w: E(G) \rightarrow \R$. 
For any edge $e$, the real number $w(e)$ is called the weight of the edge; 
for any subgraph $G'$ of  $G$, we denote by  $w(G')$ the sum of the weights 
of the edges of $G'$. 
If the weights of all the edges of $G$ are positive (respectively nonnegative), 
we say that the graph is {\bf positive-weighted} (respectively {\bf 
nonnegative-weighted}). If the weights of the internal edges are positive we say that the graph is {\bf internal-positive-weighted}, where an edge $e$ is said internal if there exists a path  with endpoints of degree greater than $2$ and containing $e$.

Throughout the paper we will consider only simple finite  connected graphs.

\begin{defin}
Let ${\cal G}=(G,w) $ be a positive-weighted graph. 
For any  $k$-subset  $S$ of $ V(G)$ with $k \geq 2$,
 we define $$ D_S({\cal G}) 
= min 
\{w(R) | \; R \text{ a connected subgraph of } G  \text{ such that } S \subset V(R) \}.$$ 
We call    $ D_S({\cal G})$  a {\bf multiweight}
 of ${\cal G}$ or, more precisely, a $k$-weight  of ${\cal G}$.
 If $R$ is a connected subgraph of $G$ such that $S \subset V(R)$ and $w(R) = D_S ({\cal G})$, we say that $R$ realizes $D_S({\cal G})$; observe that $R$ is necessarily a tree.  
 For simplicity, we denote $D_{\{i_1,..., i_k\}}({\cal G})$ by  $D_{i_1,..., i_k}({\cal G})$. 
\end{defin}

We can wonder when a family of positive real numbers is the family of multiweights of some graph or some tree.
Let  $n \in \N_{\geq 2}$ (the set of the natural numbers greater than or equal to $2$) and $A$ be a subset of $\{S \subset \{1,...., n\}| \; \sharp S \geq 2\}$. A family of positive real numbers
 $\{D_S\}_{S \in A}$  is said {\bf p-graphlike} (respectively  {\bf nn-graphlike}, {\bf ip-graphlike}) if 
there exists a positive-weighted (respectively nonnegative, internal-positive) graph  ${\cal G}=(G,w)$ with $\{1,...., n\} \subset V(G)$ 
such that $ D_{S}({\cal G}) = D_{S}$  for any 
$ S \in A$. If so, we  say that ${\cal G}$ {\bf realizes} the family $\{D_S\}_{S \in A}$.
The vertices $1,...,n$ are called {\bf labelled}.
 In the case the graph is a tree, we speak of
{\bf p-treelike} families, {\bf nn-treelike} families, {\bf ip-treelike} families. Finally, if there exists a positive-weighted tree (respectively a nonnegatve-weighted tree, internal-positive-weighted tree) ${\cal T}=(T,w)$ realizing the family and such that $\{1,...,n\} \subset L(T)$, we say that the family is {\bf p-l-treelike} (respectively  {\bf nn-l-treelike},   {\bf ip-l-treelike}).
Observe that a family of positive real numbers $\{D_S\}_{S \in A}$ is p-treelike if and only if it is nn-l-treelike.

Weighted graphs have applications in several disciplines, such as 
biology, psychology, archeology, engineering. Phylogenetic trees are weighted trees
whose vertices represent  species and the weight of an edge is given 
by how much the DNA sequences of the species represented by the vertices of the edge differ.  Weighted trees are used also to represent the evolution of languages or of manuscripts.
Weighted graphs can  represent hydraulic webs or
railway webs
where the weight of a line  is 
  the difference between the  earnings and the cost of the line or the 
  length of the line.
  It can be interesting, given a family ${\cal F}$ of real numbers,
 to wonder if there exists 
a weighted tree or, more generally, a weighted graph, with ${\cal F}$ as family of  multiweights.

There are several results about 
 families of $k$-weights of weighted graphs or trees with fixed $k$. One of the first is due to Hakimi and 
Yau: in 1965, they observed that a family of positive real numbers,
$\{D_{I}\}_{I \subset \{1,...,n\}, \sharp I = 2}$,
is p-graphlike if and only if the $D_I$ satisfy the 
triangle inequalities (see \cite{H-Y}).

In the same years, also a criterion for a family $\{D_{I}\}_{I \subset \{1,...,n\}, \sharp I = 2}$ to be p-treelike
was established, see \cite{B}, \cite{SimP}, \cite{Za}: 

\begin{thm} \label{Bune}
Let
$\{D_{I}\}_{I \subset \{1,...,n\}, \sharp I = 2}$ be a family of positive real numbers 
satisfying the triangle inequalities.
It is p-treelike (or nn-l-treelike) if and only if it satisfies the so-called $4$-point condition:

for all $a,b,c,d  \in \{1,...,n\}$,
the maximum of $$\{D_{a,b} + D_{c,d},D_{a,c} + D_{b,d},D_{a,d} + D_{b,c}
 \}$$ is attained at least twice. 
\end{thm}

For higher $k$ the literature is more recent. In \cite{P-S} Pachter and Speyer explained that the study of $k$-weights, with $k>2$, is important  beacause they are 
statistically more reliable than $2$-weights (see also \cite{SS2}).
   Moreover, they obtained an important result
about $k$-weights of positive-weighted trees with $n$ labeled leaves and $k \leq \frac{n+1}{2}$:

\begin{thm} {\bf (Pachter-Speyer)}. Let $ k ,n  \in \mathbb{N}$ with
 $3 \leq  k \leq  \frac{n+1}{2}$.  A positive-weighted tree
 ${\cal T}$ with leaves $1,...,n$ and no vertices of degree 2
is determined by the values $D_I({\cal T})$, where $ I \subset \{1,...,n\}, \sharp I = k $.
\end{thm}

Later, the study of the families of $k$-weights of weighted trees produced several other  results, see for example \cite{H-H-M-S}, \cite{L-Y-P}, \cite{B-R2} or \cite{B-R3}.  

The results we have quoted  are about families of $k$-weights with fixed $k$, but it can be interesting also 
to characterize the families of multiweights, that is the families of $k$-weights with $k$ varying in $\N_{\geq 2}$, of weighted graphs or trees. Some results in this sense are due to Bryant and Tupper, who,  in  \cite{B-T} and in \cite{B-T2}, discovered important properties regarding the  concept of diversity:

\begin{defin} A {\bf diversity} is  a pair $(X, \delta)$ where $X$ is a set and
$\delta$ is a function from the finite subsets of $X$ to $\R$ satisfying the following conditions:

\smallskip
(1) $\delta(A) \geq 0$ for any finite $A \subset X$; moreover  $\delta(A) = 0$ if and only if $\sharp A \leq 1$;

\smallskip
(2)  if $B \neq \emptyset$, then $\delta(A \cup C) \leq  \delta (A \cup B) + \delta(B \cup C)$
\smallskip
for all finite $A,B,C \subset X.$
\end{defin} 

It is easy to prove that conditions (1) and (2) imply that, if $A \subset B$, then $\delta (A) \leq \delta (B)$. Obviously, if ${\cal G}=(G,w)$ is a positive-weighted graph and we consider the pair $({\cal P}, \delta)$ where 
$${\cal P}=\{S \subset V(G) \, |\, \sharp S \geq 2\}$$
and, for any $S \in {\cal P}$,
$$\delta (S)=D_{S}({\cal G}),$$ 
we have that  $({\cal P}, \delta)$ is a diversity.

In this paper we study families of $k$-weights of positive-weighted graphs with $k$ varying in $ \N_{\geq 2}$. Precisely,
let $$\{D_S\}_{S \subset \{1,...,n\}, \; \sharp S \geq 2}  $$ be a family of
 positive real numbers; we find  necessary and sufficient conditions 
for it to be the family of multiweights of a positive weighted-graph or a positive-weighted tree with vertex set $\{1,..., n\}$, see respectively Theorem
\ref{thm1} and Theorem \ref{thm2}. 
Our results  are based on a proposition (Proposition \ref{firstprop}) 
that   relates the $k$-weights  of  a positive-weighted graph or tree for $k \geq 3$ to the $2$-weights. 
Finally we study the analogous problem  for nonnegative-weighted trees with set of leaves equal to $\{1,...,n\}$ (see Theorem \ref{thm3}).

\section{Notation}

\begin{nota} \label{notainiziali}

$ \bullet $ Throughout the paper, 
let $n \in \N $ with $ n \geq 2$; we denote by  $[n]$ the set $ \{1,..., n\}$.

$ \bullet $ For any set $S$ and $k \in \mathbb{N}$,  let ${S \choose k}$
be the set of the $k$-subsets of $S$ and let ${S \choose \geq k}$
be the set of the subsets of $S$ of cardinality greater than or equal to $k$.

$\bullet$ The words ``graph'' and ``tree''  denote respectively a finite graph and a finite tree.

 $\bullet $ 
Let  $\{D_{I}\}_{I \in A} $ with $A \subset {[n] \choose \geq 2}$ be a family of  real numbers. 
 For simplicity, we denote $D_{\{i_1,..., i_k\}}$ by  $D_{i_1,..., i_k}$. 
\end{nota}

\begin{nota} \label{cherries} 
Let T be a  tree.

$\bullet$ A {\bf node} of T is a vertex of degree greater than 2.

$\bullet$  Let $F$ be a leaf of $T$. Let $N$ be the node 
 such that the path $p$ between $N$ and $F$ does not contain any node apart from $N$. We say that $p$ is the {\bf twig} associated to $F$.
We say that an edge is {\bf internal} if it is not an edge of a twig.
It is easy to see that this definition is equivalent to the one we have given in the introduction.
We denote by $\mathring{E}(T)$ the set of the internal edges of $T$.

$\bullet $ We say that T is {\bf essential} if it has no vertices of degree $2$. 

$\bullet $ If $a$ and $b$ are vertices of T, we denote by $p(a,b)$ the path between $a$ and $b$.

$\bullet$ Let $S$ be a subset of $L(T)$. We denote by $T|_S$ the minimal subtree 
of $T$ whose vertex set   contains $S$. If ${\cal T}= (T,w)$ is a weighted tree, we denote by 
${\cal T}|_S$  the tree $T|_S$ with the weighting induced by $w$.
Let $\tilde{E} (T|_S) = \mathring{E}(T) \cap E(T|_S)$.
Observe that in general $\tilde{E} (T|_S) \neq \mathring{E} (T|_S)$, see Figure \ref{interni} for an example.

\begin{figure}[h!]
\begin{center}

\begin{tikzpicture} 
\draw [thick] (-2,0) --(2,0);
\draw [thick] (-1,0) --(-1,1);
\draw [thick] (0,0) --(0,1);
\draw [thick] (1,0) --(1,-1);
\node[below] at (-2, 0) {1};
\node[below] at (-1,0 ) {2};
\node[below] at (0,0 ) {3};
\node[above] at (1,0 ) {4};
\node[above] at (2, 0) {5};
\node[above] at (-1,1) {6};
\node[above] at (0,1) {7};
\node[below] at (1,-1) {8};
\end{tikzpicture}

\caption{Let $T$ be the tree in the figure and let $S=\{1,5,7,8\}$. The edge $\{2,3\}$ is in $\tilde{E}(T|_S)$ but not in $\mathring{E}(T|_S)$.   }
\label{interni}
\end{center}
\end{figure}
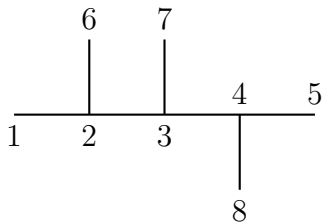

$\bullet $ 
We say that two leaves  $i$ and $j$ of $T$ are {\bf neighbours}
if in $p(i,j)$ there is only one node; 
furthermore, we say that  $C \subset L(T)$ is a {\bf cherry} if any $i,j \in C$ are neighbours.
The {\bf stalk} of a cherry is the unique node in the path 
with endpoints any two elements of the cherry.

$\bullet $ 
Let  $a,b,c,d \in L(T)$. We say that $ \langle
a, b | c, d \rangle $ holds if  in  $T|_{\{a,b,c,d\}}$ we have that $a$ and $b$ are neighbours, 
 $c$ and $d$ are neighbours, and  $a$ and $c$ are not neighbours; in this case we denote by  $\gamma_{a,b,c,d}$ the path 
between the stalk $s_{a,b}$ of $\{a,b\}$ and the stalk  $s_{c,d}$ of $ \{c,d\}$ in $T|_{\{a,b,c,d\}}$; we call it the {\bf bridge} of the $4$-subset $\{a,b,c,d\}$. The symbol 
$ \langle a,b \,| \, c, d \rangle $  is called {\bf Buneman's index} of $a,b,c,d$. 
\end{nota}

\section{Graphs and trees with all the vertices  labelled}

\begin{defin} For any  $X \subset {[n] \choose 2}$, 
 we define $G_{X}$ to be the graph such that $E(G_{X})=X$ and $V(G_{X})=\cup_{I \in X} I$.
\end{defin}

\begin{prop}\label{firstprop}
Let ${\cal G}=(G,w)$ be a positive-weighted graph with $V(G)=[n]$. For any $S \in {[n] \choose \geq 3}$, 
 we have:
$$D_{S}({\cal G})= 
\min_{
\stackrel{
 X \subset  {[n] \choose 2} \; \mbox{\scriptsize s.t.  $G_X$    tree  } }{ \mbox{\scriptsize  and $L(G_X) \subset S \subset V(G_X)$}}
}
\left\{ \sum_{I \in X} D_{I}({\cal G}) \right\}.
$$
\end{prop}

\begin{proof} Let us fix $S \in {[n] \choose \geq 3}$ and
 define 
$$ R = \left\{X \subset  {[n] \choose 2} | \;\; G_X \textrm{  a tree and } 
L(G_X) \subset S \subset V(G_X) \right\}.$$ 
Let $T$ be a connected subtree of $G$ realizing $D_S({\cal G})$.
Then, obviously, 
\begin{equation}\label{DS} 
D_{S}({\cal G})= w(T) = \sum_{I \in E(T)} w(I). 
\end{equation}

We want to prove that
\begin{equation} \label{wI}
w(I)=D_{I}({\cal G})
\end{equation}
for any $I \in E(T)$.
Obviously, 
$w(I) \geq D_I({\cal G})$; moreover,  
if, contrary to our claim, we had that  $w(I) > D_I({\cal G})$ for some $I \in E(T)$, then it would exist a path $p_I$ from one element of $I$ to the other, different from the 
edge $I$, and  such that $D_I({\cal G})=w(p_I)$; let $H$ be  the connected subgraph of $G$ obtained from $T$ replacing the edge $I$ with $p_{I}$;  we would have that 
$S\subset V(H)$ and  $w(H)<w(T)=D_{S}({\cal G})$, which is absurd.
From (\ref{DS}) and (\ref{wI})  we get
$$
D_{S}({\cal G})= \sum_{I \in E(T)} D_{I}({\cal G}).$$
From the equation above and the fact that  $E(T) \in R$, we get that 
$$D_{S}({\cal G}) \geq 
\min_{X \in R}
\left\{ \sum_{I \in X} D_{I}({\cal G}) \right\}.
$$


Let us prove the other inequality; suppose $X \in R$;  for any $I \in X$, let $p_I$ be a path in ${\cal G}$ realizing $D_{I}({\cal G})$. We have that
$$ D_{S}({\cal G})\leq w \left( \bigcup_{I \in X} p_I \right) \leq  \sum_{I \in X} w(p_I) = \sum_{I \in X} D_{I}({\cal G}),$$
where the first inequality holds because  $\bigcup_{I \in X} p_I$ is a connected graph and its vertex set  contains $S$. 
\end{proof}

\begin{thm}\label{thm1}
Let $\{D_{I}\}_{I \in {[n] \choose \geq 2 }}$ be a family of positive real numbers. There exists a positive-weighted graph ${\cal G}=(G,w)$, with $V(G)=[n]$, such that $D_I({\cal G})=D_I$ for any $I \in {[n] \choose \geq 2}$ if and only if the following two conditions hold:

\begin{itemize}
\item[$(i)$] $D_{i,j} \leq D_{i,k}+D_{j,k}$ for any $i,j,k \in [n]$;

\item[$(ii)$] for any $S \in {[n] \choose \geq 3}$ we have that 
$$
D_{S}=
\min_{
\stackrel{
 X \subset  {[n] \choose 2} \; \mbox{\scriptsize s.t.  $G_X$    tree  } }{ \mbox{\scriptsize  and $L(G_X) \subset S \subset V(G_X)$}}
}
\left\{ \sum_{I \in X} D_{I} \right\}.
$$
\end{itemize}

\end{thm}

\begin{proof} $\Longrightarrow$
Suppose that there exists a positive-weighted graph ${\cal G}=(G,w)$, with $V(G)=[n]$, such that $D_I({\cal G})=D_I$ for any $I \in {[n] \choose \geq 2}$. It is well known and easy to prove  that condition $(i)$ holds. Condition $(ii)$ follows from  Proposition \ref{firstprop}.

$\Longleftarrow$
Let $\{D_{I}\}_{I \in {[n] \choose \geq 2 }}$ be a family of positive real numbers satisfying conditions $(i)$ and $(ii)$; we can construct a positive-weighted graph ${\cal G}=(G,w)$ in the following way: let $G$ be the complete graph  with $n$ vertices, and let the weight of the edge  $\{i, j\}$ be equal to $D_{i,j}$. We have that $D_I ({\cal G})=D_I$ for any $I \in {[n] \choose 2}$ by (i). We have to prove that $D_{S}({\cal G})=D_{S}$ for any $S \in {[n] \choose \geq 3}$. By Proposition \ref{firstprop} we know that 
$$ D_{S}({\cal G})= \min_{X \in R} \left\{ \sum_{I \in X} D_{I}({\cal G}) \right\}, $$
where $R=\{X \subset {[n] \choose 2} \, | \,
\;\; G_X \textrm{  a tree and } L(G_X) \subset S \subset V(G_X)
\}$, and by assumption (ii) we have that 
$$D_{S}= \min_{X \in R} \left\{ \sum_{I \in X} D_{I} \right\}.$$
So we get the desired result, since we have already proved that $D_{I}({\cal G})=D_I$ for any $I \in {[n] \choose 2}$.
\end{proof}


Now, we want  to characterize  the families of  multiweights of positive-weighted trees.  First, we need to introduce a  definition  and to state a  theorem characterizing the families $\{D_I \}_{I  \in {n \choose 2}}$ that  are the families of $2$-weights of 
positive-weighted trees with $[n]$ as vertex set (that is, with all the vertices labelled).
The theorem, probably well-known to experts,   was suggested to us by an anonymous referee  in October 2014 as a simplification of a similar but more complicated criterion;   later we have found it also
in \cite{H-F}; we give here a shorter proof.

\begin{defin}\label{medianfamily}
Let $ n \geq 3$ and let $\{D_I\}_{I \in {[n] \choose 2}}$ be a set of positive real numbers. We say that the family $\{D_I\}$ is a {\bf median family} if, for any $a,b,c \in [n]$, there exists a unique element $m \in [n]$ such that 
$$D_{i,j}=D_{i,m}+D_{j,m}$$
for any distinct $i,j \in \{a,b,c\}.$
\end{defin}

Observe that a median family satisfies the triangle inequalities.

\begin{thm}\label{thmmedian} Let $ n \geq 3$ and
let $\{D_I\}_{I \in {[n] \choose 2}}$ be a family of positive real numbers. There exists a positive-weighted tree ${\cal T}=(T,w)$, with $V(T)=[n]$, such that $D_I(\cal{T})=D_ I$ for all $I \in {[n] \choose 2}$ if and only if the $4$-point condition holds and  the family $\{D_I\}_I$ is median.
\end{thm}

\begin{proof} 
$\Longrightarrow $ 
Obvious. 

$\Longleftarrow $
  Since the $4$-point condition holds, then, by Theorem \ref{Bune}, 
there exists a positive-weighted tree $\cal{T}=(T,w)$ such that $[n] \subset V(T)$ and $D_I(\cal{T})=D_I$ for all $I \in {[n] \choose 2}$. Obviously we can suppose that the vertices of $T$ of degree $1$ or $2$ are elements of $[n]$. 
We want to prove that $V(T)=[n]$. Let $m \in V(T)$ such that $deg(m)\geq3.$ This implies that there exist three distinct leaves, $a,b,c \in[n]$, such that $T|_{\{a,b,c\}}$ is a star with center the vertex $m$. We have that:
$$D_{i,j}(\cal{T})=D_{i,m}(\cal{T})+D_{j,m}(\cal{T})$$
for any $i,j \in \{a,b,c\}.$ Moreover, by assumption, there exists a unique element $z \in [n]$ such that:
\begin{equation} \label{med}
D_{i,j}=D_{i,z}+D_{j,z}
\end{equation}
for any $i,j \in \{a,b,c\},$ that is,
$$D_{i,j}(\cal{T})=D_{i,z}(\cal{T})+D_{j,z}(\cal{T})$$
for any $i,j \in \{a,b,c\}.$

Then, $z=m$; in fact, if $z \in V(T|_{\{a,b,c\}})$ and $z \neq m$, then $z$ would be a vertex of   the path between $m$ and one of the three leaves, suppose for example $a$; this would imply that
$D_{b,c}\neq D_{b,z}+D_{c,z},$
which is absurd by (\ref{med}). On the other hand, if $z \notin V(T|_{\{a,b,c\}})$, similarly we would obtain that 
$D_{i,j}\neq D_{i,z}+D_{j,z}$
for any $i,j \in \{a,b,c\},$ which is absurd. So $z=m$; therefore   $m \in [n]$, as we wanted to prove.
\end{proof}

\begin{thm}\label{thm2}
Let $ n \geq 3$ and let $\{D_I\}_{I \in {[n] \choose \geq 2}}$ be a family of positive real numbers. There exists a positive-weighted tree ${\cal T}=(T,w)$, with $V(T)=[n]$, such that $D_I({\cal T})=D_I$ for any $I \in {[n] \choose \geq 2}$ if and only if the following three conditions hold:

\begin{itemize}

\item[$(i)$] the family $\{D_I\}_{I \in {[n] \choose 2}}$ satisfies the $4$-point condition;
\item[$(ii)$] the family $\{D_I\}_{I \in {[n] \choose 2}}$ is median;
\item[$(iii)$] for any $S \in {[n] \choose \geq 3}$,  we have:
\begin{equation} \label{eq:thm3}
D_{S}= 
\min_{
\stackrel{
 X \subset  {[n] \choose 2} \; \mbox{\scriptsize s.t.  $G_X$    tree  } }{ \mbox{\scriptsize  and $L(G_X) \subset S \subset V(G_X)$}}
} 
\left\{ \sum_{I \in X} D_{I} \right\}.
\end{equation}
 
\end{itemize}

\end{thm}

\begin{proof}
$\Longrightarrow $
 Condition $(i)$ is satisfied by Theorem \ref{Bune},
 condition $(ii)$ is satisfied by Theorem \ref{thmmedian}  and, finally, Proposition \ref{firstprop} assures us of the last condition. 
 
 $\Longleftarrow$
 Let $\{D_I\}_{I \in {[n] \choose \geq 2}}$ be a family of positive real numbers satisfying conditions (i), (ii) and (iii). Suppose that ${\cal T}=(T,w)$ is a positive-weighted tree such that $V(T)=[n]$ and $D_{I}({\cal T})=D_I$ for any $I \in {[n] \choose 2}$ (such a tree exists by Theorem \ref{thmmedian} and conditions (i) and (ii)); we can prove that $D_{S}({\cal T})= D_{S}$ for any $S \in {[n] \choose \geq 3}$ arguing  as in the proof of Theorem \ref{thm1}.
\end{proof}

\section{Trees with labels only on the leaves}

In order to study the families of multiweights of nonnegative-weighted trees with set of leaves equal to $[n]$, we need  to recall some notation and some facts from \cite{B-R1}.

\begin{defin}
 Let $T$ be a tree and let $a,b,c,d ,x \in L(T)$.
Let $S$ be a subtree of $T|_{a,b,c,d}$. 

Let $\tilde{x}$ be the vertex such that $p(x,\tilde{x}) 
$ is the minimal path whose union with  $T|_{a,b,c,d}$ is connected;
we say that   $x$ {\bf clings} to $T|_{a,b,c,d}$ in $S$ if 
$\tilde{x} \in V(S) $.


\end{defin}

See Figure \ref{clings} for an example: let $T$ be the tree in the 
figure and let $S=p(a,b)$.
\begin{figure}[h!]
\begin{center}

\begin{tikzpicture} 
\draw [thick] (0,0) --(1,0);
\draw [thick] (0,0) --(-1,0);
\draw [thick] (1,0) --(2,1);
\draw [thick] (1,0) --(2,-1);
\draw [thick] (-1,0) --(-2,1);
\draw [thick] (-1,0) --(-2,-1);
\draw [thick] (-1.5,-0.5) --(-2.5,-0.8);
\draw [thick] (1.3,0.3) --(1.4,1);
\draw [thick] (-0.5,-0.8) --(-0.2,0);
\draw [thick] (-0.1,-0.8) --(-0.3,-0.3);
\node[above] at (-2, 1) {$a= \tilde{a}$};
\node[below] at (-2,-1 ) {$b= \tilde{b}$};
\node[above] at (2.3, 1) {$c=\tilde{c}$};
\node[below] at (2.3,-1 ) {$d= \tilde{d}$};
\node[left] at (-2.5,-0.8) {$x$};
\node[right] at (-1.5,-0.5) {$\tilde{x}$};
\node[below] at (0,-0.8) {$z$};
\node[above] at (-0.2,0) {$\tilde{z}$};
\end{tikzpicture}

\caption{the leaves $x,a,b$ cling   to $T|_{a,b,c,d}$ in  $S:=p(a,b)$, while $z,c,d$ do not cling  to $T|_{a,b,c,d}$ in  $S$}
\label{clings}
\end{center}
\end{figure}
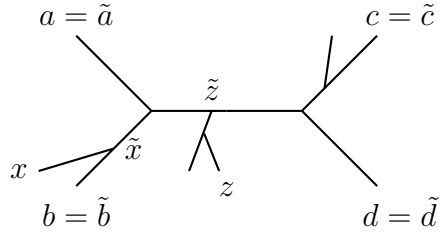

\begin{defin} \label{defL} Let $\{D_I\}_{ I \in { [n] \choose 2}}$be a family
of real numbers.  For any distinct $a,b,c,d \in [n]$, 
let 
$$ L^{\{a,b,c,d\}}_{\{a,b\}} =\left\{x \in [n]-\{a,b,c,d\} \; \bigg| \;   \begin{array}{ll} \mbox{ either } & D_{x,z} - D_{a,z}  \mbox{ does not depend on } z \in \{b,c,d\}  
\\ \mbox{ or } &   D_{x,z} - D_{b,z}  \mbox{ does not depend on } z \in 
\{a,c,d\} \end{array}
\right\} \cup \{a,b\}.$$
We will denote $ L^{\{a,b,c,d\}}_{\{a,b\}}$ simply by $ L^{a,b,c,d}_{a,b}$ and 
we will omit the superscript  when the $4$-set 
which we are referring to is clear from the context.
\end{defin}

The definition above seems rather  obscure, but the following proposition and example will clarify it.

\begin{prop} \label{bridge}
Let ${\cal T} =(T,w) $ be an essential internal-positive-weighted tree. Denote  $D_{i,j} ({\cal T})$ by  $D_{i,j} $  for distinct $i,j \in L(T)$. Let $a,b,c,d \in L(T)$. 

1) If $\langle a,b \, |\, c,d \rangle $  holds,  we have that $L^{a,b,c,d}_{a,b}$ 
is the set of the elements $x$ of $L(T)$ clinging
 to $T|_{a,b,c,d}$  in  $p(a,b)$ and
$L^{a,b,c,d}_{c,d}$ is the set of the elements 
$x$ of $L(T)$ clinging  to $T|_{a,b,c,d}$  in  $p(c,d)$. 

2) We have that 
$\langle a,b \, |\, c,d \rangle $ holds and  the bridge of $(a,b,c,d)$ is given by exactly one   edge if and only if the following conditions hold:

(i)  $$D_{a,b} + D_{c,d} <
 D_{a,c} + D_{b,d} = D_{a,d}+ D_{b,c};$$ 

(ii) $L_{a,b} \cup L_{c,d}=  L(T)$.

\end{prop}

See \cite{B-R1} for the proof.

\bigskip

{\bf Example.} 
  Let ${\cal T}$ be the
tree represented in Figure \ref{es} with all the weights of the edges equal to $1$;
 consider the $4$-set $\{1,3,4,7\}$; we have that   $ \langle 1, 3 \,|\, 4, 7 \rangle $ holds, $L_{1,3}^{1,3,4,7}=\{1,2,3,9,10\}$ and $L_{4,7}^{1,3,4,7}=\{4,5,6,7,8\}$, so $L_{1,3}^{1,3,4,7} \cup L_{4,7}^{1,3,4,7}=[10]$ and $\gamma_{1,3,4,7}$ is composed by only one edge. Now consider the $4$-set $\{1,9,4,7\}$; we have that  $ \langle 1, 9 \,|\, 4, 7 \rangle $ holds, $L_{1,9}^{1,9,4,7}=\{1,2,9,10\}$ and $L_{4,7}^{1,9,4,7}=\{4,5,6,7,8\}$, so $L_{1,9}^{1,9,4,7} \cup L_{4,7}^{1,9,4,7} \neq [10]$ and in fact $\gamma_{1,9,4,7}$ is composed by more than one edge.

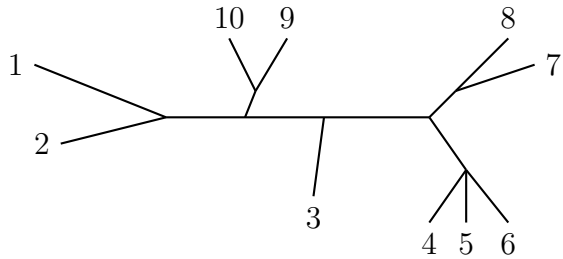
\begin{figure}[h!]
\begin{center}
\begin{tikzpicture}[scale=0.7]
\draw [thick] (0,0) --(5,0) ; 
\draw [thick] (0,0) --(-2.5,1);
\draw [thick] (0,0) --(-2,-0.5);
\draw [thick] (5,0) --(5.5,0.5); 
\draw [thick] (7,1) --(5.5,0.5); 
\draw [thick] (6.5,1.5) --(5.5,0.5); 
\draw [thick] (5,0) --(5.7,-1);
\draw [thick] (5,-2) --(5.7,-1);
\draw [thick] (6.5,-2) --(5.7,-1);
\draw [thick] (5.7,-2) --(5.7,-1);
\draw [thick] (1.5,0) --(1.7,0.5) ; 
\draw [thick] (1.2,1.5) --(1.7,0.5) ; 
\draw [thick] (2.3,1.5) --(1.7,0.5) ; 
\draw [thick] (3,0) --(2.8,-1.5) ;

\node [left] at (-2.5,1) {$1$};
\node [left] at (-2,-0.5) {$2$};
\node [above] at (6.5,1.5) {$8$};
\node [above] at (2.3,1.5) {$9$};
\node [above] at (1.2,1.5) {$10$};
\node [below] at (2.8,-1.5) {$3$};
\node [below] at (5,-2) {$4$};
\node [below] at (5.7,-2) {$5$};
\node [below] at (6.5,-2) {$6$};
\node [right] at (7,1) {$7$};

\end{tikzpicture}

\caption{An example to explain Proposition \ref{bridge}} \label{es}
 
\end{center}
\end{figure}

\begin{defin}\label{qpiccolo}
Let  $\{D_{I}\}_{I \in {[n] \choose \geq 2}}$ be a family of positive real numbers.
Let us define
$$Q= \left\{(a,b,c,d)  \textrm{ ordered } 4\textrm{-subset of } [n]   \;| \;  \;\; \begin{array}{l}
D_{a,b} + D_{c,d}
< D_{a,c} + D_{b,d}= D_{a,d} + D_{b,c},  \vspace*{0.1cm} \\
L_{a,b} \cup L_{c,d}=[n]
\end{array}
  \right\}/ \sim,$$ where 
$(a,b,c,d) \sim  (a',b',c',d')$ if and only if
$\{ L_{a,b}, L_{c,d}\}
=\{ L_{a',b'} ,  L_{c',d'}\}$.

Moreover, for any $S \in { [n]  \choose \geq 4}$, 
let $$Q(S)=\{ [a,b,c,d]  \in Q\; | \; S \cap L_{a,b} \neq \emptyset, \; 
 S \cap L_{c,d} \neq \emptyset\}.$$
If $[(a,b,c,d)] \in Q(S)$,  we define:
$$q(S)_{[(a,b,c,d)]} = \sharp (L^{\{a,b,c,d\}}_{\{a,b\}} \cap S) \cdot \sharp (L^{\{a,b,c,d\}}_{\{c,d\}} \cap S).$$
\end{defin}

\begin{rem} \label{internal}
Let ${\cal T} =(T,w) $ be an essential internal-positive-weighted tree with $L(T)=[n]$. Denote  $D_{i,j} ({\cal T})$ by  $D_{i,j} $  for distinct $i,j \in L(T)$.
Observe that, by Proposition \ref{bridge}, the set $Q$ is in bijection with $ \mathring{E}(T)$ and, for any $S \in { [n]  \choose \geq 4}$,  the set $Q(S)$ is in bijection with $ \tilde{E}(T|_S)$.
\end{rem}

\begin{prop}\label{lastprop}
Let ${\cal T} =(T,w) $ be a nonnegative-weighted tree with $L(T)=[n]$.
 Let us denote  $D_{i,j} ({\cal T})$ by  $D_{i,j} $  for distinct $i,j \in L(T)$.
 For any $S \subset [n]$ we have that
\begin{itemize}
\item if $\sharp S = 3$ then 
\begin{equation}\label{eq:primatree}
D_S({\cal T})= \frac{1}{2}\sum_{ \{ i,j \}  \in {S  \choose 2}} D_{i,j};
\end{equation}

\item if $\sharp S \geq 4$ then 
\begin{equation}\label{eq:secondatree}
D_S({\cal T})= \frac{1}{\sharp S -1}\,  \left[ \sum_{  \{ i,j \}  \in {S  \choose 2} } D_{i,j} - \sum_{[(a,b,c,d)] \in Q(S)} \frac{ D_{a,c}+D_{b,d}-D_{a,b}-D_{c,d}}{2} \; \big(q(S)_{[(a,b,c,d)]}+1-\sharp S\big) \right] ,
\end{equation}
\end{itemize}

where $Q(S)$ and $q(S)_{[(a,b,c,d)]}$ are defined as in Definition \ref{qpiccolo}.
\end{prop}

\begin{proof} Obviously we can suppose that $T$ is essential and internal-positive-weighted.
If $\sharp S = 3$, the minimal subtree of $T$ containing the elements of $S$ is a star with three leaves, so it is easy to check  (\ref{eq:primatree}). If $\sharp S \geq 4$, for any $i \in S$, call $e_i$ the corresponding twig, which is composed by only one edge because $T$ is essential; by definition we have that 
$$D_{S}({\cal T})= \sum_{i \in S} w(e_i) + \sum_{e \in \tilde{E}(T|_{S}) } w(e);$$ 
thus, by Remark \ref{internal}, 
\begin{equation} \label{br}
D_{S}({\cal T})= \sum_{i \in S} w(e_i) + \sum_{[a,b,c,d] \in Q(S) } \frac{D_{a,c} + D_{b,d} - D_{a,b} -D_{c,d}}{2}.
\end{equation}
Moreover 
$$\sum_{ \{ i,j \}  \in {S  \choose 2}} D_{i,j} = (\sharp S-1)  \sum_{i \in S} w(e_i) 
+  \sum_{[a,b,c,d] \in Q(S) } q(S)_{[a,b,c,d]}  \frac{D_{a,c} + D_{b,d} - D_{a,b} -D_{c,d}}{2} ,$$
because $q|_{[(a,b,c,d)]}$ is the number of the $2$-subsets $\{i,j\}$  of $S$ such that $i$ and $j$ belong to  different connected components of $T \smallsetminus \{e\}$  (that is, $e$ belongs to the path $p(i,j)$ realizing $D_{i,j}$); hence 
\begin{equation} \label{brr} \sum_{i \in S} w(e_i)  = \frac{1}{  (\sharp S-1) }
\sum_{ \{ i,j \}  \in {S  \choose 2}} D_{i,j} 
- \sum_{[a,b,c,d] \in Q(S) }  q(S)_{[a,b,c,d]}  \frac{D_{a,c} + D_{b,d} - D_{a,b} -D_{c,d}}{2   (\sharp S-1) } .
\end{equation}
From (\ref{br}) and (\ref{brr}) we get immediately our statement.
\end{proof}

\begin{thm}\label{thm3}
Let $\{D_I\}_{I \in {[n] \choose \geq 2}}$ be a family of positive real numbers. There exists a nonnegative-weighted tree ${\cal T}=(T,w)$, with $L(T)=[n]$, such that $D_I({\cal T})=D_I$ for any $I \in {[n] \choose \geq 2}$ if and only if the following three conditions hold:

\begin{itemize}
\item[(i)] the family $\{D_I\}_{I \in {[n] \choose 2}}$ satisfies the $4$-point condition;
\item[(ii)] for any $S \in {[n] \choose 3}$,  we have:
\begin{equation}
D_S= \frac{1}{2}\sum_{  \{ i,j \}  \in {S  \choose 2} } D_{i,j};
\end{equation}
\item[(iii)] for any $S \in {[n] \choose \geq 4}$,  we have:
\begin{equation}
D_S= \frac{1}{\sharp S -1}\,  \left[ \sum_{  \{ i,j \}  \in {S  \choose 2} } D_{i,j} - \sum_{[(a,b,c,d)] \in Q(S)} \frac{D_{a,c}+D_{b,d}-D_{a,b}-D_{c,d}}{2}\; \left(q(S)_{[(a,b,c,d)]}+1-\sharp S \right) \right] ,
\end{equation}
where $Q(S)$ and $q(S)_{[(a,b,c,d)]}$ are defined as in Definition \ref{qpiccolo}.
\end{itemize}
\end{thm}

\begin{proof}
$\Longrightarrow $
 Condition (i) is satisfied by Theorem \ref{Bune}  and Proposition \ref{lastprop} assures us of (ii) and (iii). 
 
$\Longleftarrow$
Suppose $\{D_I\}$ satisfies conditions (i), (ii) and (iii).  Let ${\cal T}=(T,w) $ be a nonnegative-weighted tree 
 such that $L(T)=[n]$ and $D({\cal T})= D_I $ for every $I  \in {[n] \choose 
 2} $ (it exists by Theorem \ref{Bune}). By Proposition \ref{lastprop} and
 conditions (ii) and  (iii), 
we have that $D_{S}({\cal T})= D_{S}$ for any $S \in {[n] \choose \geq 3}$.
\end{proof}

Since a family of positive real numbers is nn-l-treelike if and only if it is p-treelike, Theorem \ref{thm3} can be reformulated  in the following way:

\begin{thm}\label{thm4}
Let $\{D_I\}_{I \in {[n] \choose \geq 2}}$ be a family of positive real numbers. There exists a positive-weighted tree ${\cal T}=(T,w)$, with $[n] \subset V(T)$, such that $D_I({\cal T})=D_I$ for any $I \in {[n] \choose \geq 2}$ if and only if conditions $(i)$, $(ii)$ and  $(iii)$ of Theorem \ref{thm3} hold.
\end{thm}

\bigskip

{\bf
E-mail addresses:}
baldisser@math.unifi.it, rubei@math.unifi.it

\end{document}